\documentclass[oneside,leqno,10pt]{article}
\usepackage{amssymb,amsmath,latexsym,amsthm,stmaryrd}

\def\gh{\mathfrak{h}}

\def\gl{\mathfrak{l}}
\def\gm{\mathfrak{m}}
\def\gn{\mathfrak{n}}
\def\go{\mathfrak{o}}
\def\gp{\mathfrak{p}}

\def\gs{\mathfrak{s}}

\def\gu{\mathfrak{u}}
\def\gv{\mathfrak{v}}
\def\gw{\mathfrak{w}}

\def\gz{\mathfrak{z}}

\def\C{\mathbb{C}}

\def\F{\mathbb{F}}

\def\H{\mathbb{H}}

\def\R{\mathbb{R}}

\def\cC{\mathcal{C}}
\def\cD{\mathcal{D}}
\def\cE{\mathcal{E}}
\def\cF{\mathcal{F}}

\def\cH{\mathcal{H}}

\def\cO{\mathcal{O}}

\def\Im{{\rm Im}\,}
\def\Re{{\rm Re}\,}

\def\Span{{\rm Span}\,}
\def\Ad{{\rm Ad}}

\def\Pf{{\rm Pf}\,}

\def\Ind{{\rm Ind\,}}

\def\Span{{\rm Span}\,}
\def\diag{{\rm diag}}

\renewcommand{\thesection}{\arabic{section}}

\renewcommand{\theequation}{\thesection.\arabic{equation}}

\newtheorem{theorem}[equation]{Theorem}

\newtheorem{lemma}[equation]{Lemma}

\newtheorem{proposition}[equation]{Proposition}

\def\sideremark#1{\ifvmode\leavevmode\fi\vadjust{\vbox to0pt{\vss
 \hbox to 0pt{\hskip\hsize\hskip1em
\vbox{\hsize2cm\tiny\raggedright\pretolerance10000 
 \noindent #1\hfill}\hss}\vbox to8pt{\vfil}\vss}}} 

\title{On the Analytic Structure of Commutative Nilmanifolds}

\author{Joseph A. Wolf\footnote{Research partially supported by the Simons
Foundation.\newline
}}
\date{June 30, 2014}

\begin{document}

\maketitle

\abstract{In the classification theorems of Vinberg and Yakimova for
commutative nilmanifolds, the relevant nilpotent groups have a very
surprising analytic property.  The manifolds are of the form 
$G/K = N\rtimes K/K$ where, in all but three cases, the nilpotent group
$N$ has irreducible unitary representations whose coefficients are 
square integrable modulo the center $Z$ of $N$.  Here we show that,
in those three ``exceptional'' cases, the group $N$ is a semidirect product
$N_1\rtimes \mathbb{R}$ or $N_1\rtimes\mathbb{C}$ where the normal 
subgroup $N_1$ contains 
the center $Z$ of $N$ and has irreducible unitary representations whose 
coefficients are square integrable modulo $Z$.  This leads directly to 
explicit harmonic analysis and Fourier inversion formulae for commutative
nilmanifolds.}

\section{Introduction}
\label{sec1}
\setcounter{equation}{0}

A commutative space $X = G/K$, or equivalently a Gelfand pair $(G,K)$,
consists of a locally compact group $G$ and a compact subgroup $K$ such
that the convolution algebra $L^1(K\backslash G/K)$ is commutative.
When $G$ is a connected Lie group it is equivalent to say that the algebra
$\cD(G,K)$ of $G$--invariant differential operators on $G/K$ is commutative.
We say that the commutative space $G/K$ is a commutative nilmanifold
if it is a nilmanifold in the sense that some nilpotent analytic
subgroup $N$ of $G$ acts transitively.  When $G/K$ is connected and
simply connected it follows that $N$ is the nilradical of $G$, that $N$
acts simply transitively on $G/K$, and that $G$ is the semidirect
product group $N\rtimes K$, so that $G/K = (N\rtimes K)/K$.  In this paper
we study commutative nilmanifolds $G/K = (N\rtimes K)/K$, examine the 
structure of $N$, and describe the consequences for harmonic analysis
on $G/K$.
\medskip

In Section \ref{sec2} we review the relevant material on commutative 
spaces, riemannian nilmanifolds, and commutative nilmanifolds.  There
we recall the Vinberg classification of irreducible
commutative nilmanifolds and the Yakimova classification of those
that satisfy certain technical conditions.  
\medskip

In Section \ref{sec3} we review the theory of square integrable and
stepwise square integrable representations of nilpotent Lie groups,
and we indicate how it applies to commutative nilmanifolds.  With three
exceptions the commutative nilmanifolds $(N\rtimes K)/K$, described in
the tables of Section \ref{sec2}, have the property that $N$ has square
integrable (modulo the center of $N$) 
representations.  We then check each of these three ``exceptional'' 
cases and verify stepwise (in fact $2$--step) square integrability for them.
\medskip

In Section \ref{sec4} we combine the results of Section \ref{sec3} with
principal orbit theory for the action of $K$ on $\gz^*$ to obtain
explicit Plancherel and Fourier Inversion theorems for our commutative
nilmanifolds.
\medskip

Finally, in Section \ref{sec5} we specialize these results to weakly
symmetric Riemannian nilmanifolds and extend that specialization to
weakly symmetric Finsler nilmanifolds.
\medskip

\section{Commutative Nilmanifolds}
\label{sec2}
\setcounter{equation}{0}

A homogeneous space $X = G/K$ is called {\em commutative}, and
the pair $(G,K)$ is called a {\em Gelfand pair}, when $G$ is a 
locally compact group, $K$ is a compact subgroup, and
the convolution algebra $L^1(K\backslash G/K)$ is commutative.
Here $L^1(K\backslash G/K)$ denotes the space of $L^1$ functions
on $G$ that satisfy $f(kxk') = f(x)$ for $x \in G$ and $k,k' \in K$,
and the composition in $L^1(K\backslash G/K)$ is the usual 
convolution $(f*h)(g) = \int_G f(x)h(x^{-1}g)d\mu_{_G}(x)$ on $G$.
In assembling the material we need on commutative spaces we will
depend on the exposition and results from \cite{W2007}.
\medskip

If $G$ is a Lie group and $K$ is a closed subgroup we write
$\cD(G,K)$ for the algebra of $G$--invariant differential operators
on $G/K$.  A theorem of Thomas \cite{T1984} says: If $G$ is a connected Lie
group and $K$ is a compact subgroup, then $(G,K)$ is a Gelfand pair
if and only if $\cD(G,K)$ is commutative.
\medskip

By {\em nilmanifold} we mean a differentiable manifold on which a nilpotent
Lie group acts transitively.  By {\em commutative nilmanifold} we mean a
commutative space $G/K$ such that $G$ is a Lie group and a closed nilpotent 
subgroup $N$ of $G$ acts transitively.  In that notation, if $G/K$ is 
simply connected then (\cite[Theorem 4.2]{W1963}, or see 
\cite[Theorem 13.1.6]{W2007})  $N$ is the nilpotent
radical of $G$, $N$ acts simply transitively on $G/K$, and $G$ is the
semidirect product $N\rtimes K$.  Further, there are several independent
proofs that $N$ is abelian or $2$--step nilpotent; see 
\cite[\S 13.1]{W2007}.
\medskip

We look at the classification for reasons that will emerge in 
Section \ref{sec3}.
\medskip

Let $G/K$ be a connected simply connected commutative nilmanifold,
$G = N \rtimes K$.
We first consider the case where $G/K$ and $(G,K)$ are
{\em irreducible} in the sense that $[\gn,\gn]$ (which must be central)
is the center of $\gn$ and $K$ acts irreducibly on $\gn/[\gn,\gn]$.
\medskip

Let $Z_G^0$ denote the identity component of the
center of $G$.  If $Z$ is a closed connected $\Ad(K)$--invariant
subgroup of $Z_G^0$, then
$(G/Z,K/Z))$ is a Gelfand pair and is called a {\em central
reduction} of $(G,K)$.  The pair $(G,K)$ is called {\em maximal} if it is
not a nontrivial central reduction.  Here is a table of all the groups
$K$ and algebras $\gn= \gz + \gv$, $\gz = [\gn,\gn]$, for irreducible
maximal Gelfand pairs $(N\rtimes K, K)$ where $N$ is a connected simply
connected nilpotent Lie group.  Here $\F$ is $\R$, $\C$ or $\H$\,\,,
$\Im \F^{s\times s}$ is the space of skew
hermitian $s \times s$ matrices over $\F$, $\Re \F^{s\times s}$ is the
space of hermitian $s \times s$ matrices over $\F$; $\Im \F_0^{s\times s}$
and $\Re \F_0^{s\times s}$ are those of trace $0$.  
The Lie algebra structure
is given by $\gv \times \gv \to \gz$ and should be clear, but is explained
in detail in \cite[\S 13.4B]{W2007}.  The result is due to E. B. Vinberg.

{\footnotesize
\begin{equation} \label{vin-table}
\begin{tabular}{|r|c|c|c|c|c|}\hline
\multicolumn{6}{| c |}{Maximal Irreducible Nilpotent Gelfand 
        Pairs $(N\rtimes K,K)$ with $\gn \ne \gz$ 
	\quad (\cite{V2001}, \cite{V2003})}\\
\hline \hline
 & Group $K$ & $\gv$ & $\gz$ & $U(1)$ & max \\ \hline
1 & $SO(n)$ & $\R^n$ & $\Lambda^2\R^n = \gs\go(n)$ &  & \\ \hline
2 & $Spin(7)$ & $\R^8 = \mathbb{O}$ & $\R^7 = \Im\mathbb{O}$  &  & \\ \hline
3 & $G_2$ & $\R^7 = \Im\mathbb{O}$ & $\R^7 = \Im\mathbb{O}$ &  & \\ \hline
4 & $U(1)\cdot SO(n)$ & $\C^n$ & $\Im\C$ & & $n\ne 4$ \\ \hline
5 & $(U(1)\cdot) SU(n)$ & $\C^n$ & $\Lambda^2\C^n \oplus\Im\C$ & 
        $n$ odd &  \\ \hline
6 & $SU(n), n$ odd & $\C^n$ & $\Lambda^2\C^n$ &  & \\ \hline
7 & $SU(n), n$ odd & $\C^n$ & $\Im\C$ &  & \\ \hline
8 & $U(n)$ & $\C^n$ & $\Im \C^{n\times n} = \gu(n)$ &  & \\ \hline
9 & $(U(1)\cdot) Sp(n)$ & $\H^n$ & $\Re \H^{n \times n}_0 \oplus \Im\H$ &  & 
        \\ \hline
10 & $U(n)$ & $S^2\C^n$ & $\R$ & & \\ \hline
11 & $(U(1)\cdot) SU(n), n \geqq 3$ & ${\Lambda}^2\C^n$ & $\R$ & $n$ even & 
        \\ \hline
12 & $U(1)\cdot Spin(7)$ & $\C^8$ & $\R^7 \oplus \R$ & & \\ \hline
13 & $U(1)\cdot Spin(9)$ & $\C^{16}$ & $\R$ & & \\ \hline
14 & $(U(1)\cdot) Spin(10)$ & $\C^{16}$ & $\R$ & & \\ \hline
15 & $U(1)\cdot G_2$ & $\C^7$ & $\R$ & & \\ \hline
16 & $U(1)\cdot E_6$ & $\C^{27}$ & $\R$ & & \\ \hline
17 & $Sp(1)\times Sp(n)$ & $\H^n$ & $\Im \H = \gs\gp(1)$ & & $n \geqq 2$ 
        \\ \hline
18 & $Sp(2)\times Sp(n)$ & $\H^{2\times n}$ & 
        $\Im \H^{2\times 2} = \gs\gp(2)$ & & \\ \hline
19 & $(U(1)\cdot) SU(m) \times SU(n)$ &  &  &  & \\
   & $m,n \geqq 3$ & $\C^m\otimes \C^n$ & $\R$ & $m=n$ &   \\ \hline
20 & $(U(1)\cdot) SU(2) \times SU(n)$ & $\C^2 \otimes \C^n$ & 
        $\Im \C^{2\times 2} = \gu(2)$ & $n=2$ & \\ \hline
21 & $(U(1)\cdot) Sp(2) \times SU(n)$ & $\H^2\otimes \C^n$ & $\R$ & 
        $n \leqq 4$ & $n \geqq 3$ \\ \hline
22 & $U(2)\times Sp(n)$ & $\C^2 \otimes \H^n$ & $\Im \C^{2\times 2} = \gu(2)$ &
        & \\ \hline
23 & $U(3)\times Sp(n)$ & $\C^3 \otimes \H^n$ & $\R$ & & $n \geqq 2$
        \\ \hline
\end{tabular}
\end{equation}
}

Often one can replace $K$ by a smaller group in such a way that $(G,K)$
continues to be a Gelfand pair.  For example, in (\ref{vin-table}),
item 2, where $N$ is an octonionic Heisenberg group, the
pairs $(N\rtimes Spin(7),Spin(7))$,
$(N\rtimes Spin(6),Spin(6))$ and $(N\rtimes Spin(5),Spin(5))$ all are
Gelfand pairs; see \cite[Proposition 5.6]{L1999}.
\medskip

\noindent {\bf Notation.}
All groups are real.  If $K$ is denoted $(U(1) \cdot )L$ it can be
$U(1)\cdot L$ or $L$; where noted in the ``$U(1)$'' column it can only be
$U(1)\cdot L$.  Note $U(1)\cdot SU(n) = U(n)$.  For $SO(n)$ it is understood
that $n \geqq 3$, and for $U(n)$ and $SU(n)$ it is understood that
$n \geqq 2$.  If some pairs in the series are not maximal, the maximality
condition is noted in the ``maximal'' column.
\hfill$\diamondsuit$
\medskip

The classification of commutative nilmanifolds is based on (\ref{vin-table})
and developed by O. Yakimova (\cite{Y2005}, \cite{Y2006}). 
For an exposition see \S\S 13.4C and 13.4D, and Chapter 15, in \cite{W2007}.
In brief, Yakimova starts by defining 
technical conditions {\em principal} and {$Sp(1)$--saturated}.  Then
she classifies Gelfand pairs $(N\rtimes K, K)$ that are indecomposable, 
principal, maximal and $Sp(1)$--saturated.  See Table (\ref{vin-table-ipms})
 below.  Finally
she introduces some combinatorial methods to complete the classification.
\medskip

In Table (\ref{vin-table-ipms}), $\gh_{n;\F}$ denotes the Heisenberg algebra 
$\Im\F + \F^n$ of real dimension $(\dim_R\F - 1) + n\dim_\R\F$ with 
composition $[(z,u),(w,v)] = (\Im (\langle u, v\rangle),0)$ where
$\F$ is $\C$, $\H$ or $\mathbb{O}$.  Also in the table, $\gv = \gn/[\gn,\gn]$
and the summands in double parenthesis ((..)) are the subalgebras 
$[\gw,\gw] + \gw$ where $\gw$ is a $K$--irreducible subspace of $\gv$ 
with $[\gw,\gw] \ne 0$.  The summands not in parentheses are $K$--invariant 
subspaces of $\gw \subset \gv$ with $[\gw,\gw] = 0$.  Thus 
$\gn = [\gn,\gn] + \gv$, vector space direct sum, and its center $\gz$ is
the sum of $[\gn,\gn]$ with those summands listed for $\gv$ that are {\it not}
enclosed in double parenthesis ((..)).
\medskip

Here is Yakimova's classification of indecomposable, principal, maximal and 
$Sp(1)$--saturated commutative pairs $(N\rtimes K,K)$ where the action of
$K$ on $\gv$ is reducible.  We omit the case $[\gn,\gn] = 0$, where $N = \R^n$
and $K$ is any closed subgroup of the orthogonal group $O(n)$.
\medskip

{\footnotesize
\begin{equation} \label{vin-table-ipms}
\begin{tabular}{|r|l|l|l|l|}\hline
\multicolumn{5}{| c |}{Maximal Indecomposable Principal Saturated Nilpotent 
Gelfand Pairs $(N\rtimes K,K)$,}\\
\multicolumn{5}{| c |}{$N$ Nonabelian, Where the Action of $K$ on 
$\gv \cong \gn/[\gn,\gn]$ is Reducible \quad \cite{Y2005}, \cite{Y2006}}\\
\hline \hline
 & Group $K$ & $K$--module $\gv$ & $K$--module $[\gn,\gn]$ & Algebra $\gn$ 
        \\ \hline
1 & $U(n)$ & $\C^n \oplus \gs\gu(n)$ & $\R$ & 
        $((\gh_{n;\C})) + \gs\gu(n)$  \\ \hline
2 & $U(4)$ & $\C^4 \oplus \R^6$ & $\Im\C \oplus \Lambda^2\C^4$ & 
        $((\Im\C + \Lambda^2\C^4 +\C^4)) + \R^6$ \\ \hline 
3 & $U(1)\times U(n)$ & $\C^n \oplus \Lambda^2\C^n$   & $\R \oplus \R$   &
        $((\gh_{n;\C})) + ((\gh_{n(n-1)/2;\C} ))$ \\ \hline
4 & $SU(4)$ & $\C^4 \oplus \R^6$ & $\Im\C \oplus \Re\H^{2\times 2}$ &
        $((\Im\C + \Re\H^{2\times 2} + \C^4)) + \R^6$ \\ \hline
5 & $U(2)\times U(4)$ & $\C^{2\times 4} \oplus \R^6$ & $\Im \C^{2\times 2}$ &
        $((\Im \C^{2\times 2} + \C^{2\times 4})) + \R^6$ \\ \hline
6 & $S(U(4)\times U(m))$ & $\C^{4 \times m} \oplus \R^6$ & $\R$ &
        $((\gh_{4m;\C})) + \R^6$ \\ \hline
7 & $U(m)\times U(n)$ & $\C^{m\times n}\oplus \C^m$ & $\R \oplus \R$ &
        $((\gh_{mn;\C})) + ((\gh_{m;\C}))$ \\ \hline
8 & $U(1)\times Sp(n)\times U(1)$ & $\C^{2n}\oplus \C^{2n}$ & $\R\oplus \R$ &
        $((\gh_{2n;\C})) + ((\gh_{2n;\C}))$ \\ \hline
9 & $Sp(1)\times Sp(n)\times U(1)$ & $\H^n \oplus \H^n$ & $\Im \H \oplus \R$ &
        $((\gh_{n;\H})) + ((\gh_{2n;\C}))$ \\ \hline
10 & $Sp(1)\times Sp(n)\times Sp(1)$ & $\H^n\oplus\H^n$ & $\Im\H\oplus\Im\H$ &
        $((\gh_{n;\H})) + ((\gh_{n;\H}))$ \\ \hline
11 & $Sp(n)\times\{Sp(1),U(1),\{1\}\}$ & $\H^n\oplus\H^{n\times m}$ & $\Im\H$ &
        $((\gh_{n;\H})) + \H^{n\times m}$ \\ 
   & \phantom{XXXXX} $\times Sp(m)$ & & & \\ \hline
12 & $Sp(n)\times\{Sp(1),U(1),\{1\}\}$ & $\H^n \oplus \Re \H^{n\times n}_0$ &
        $\Im \H$ & $((\gh_{n;\H})) + \Re \H^{n\times n}_0$ \\ \hline
13 & $Spin(7)\times \{SO(2),\{1\}\}$ & $(\R^8=\mathbb{O}) \oplus \R^{7\times 2}$ & 
        $\R^7 = \Im\mathbb{O}$ &
        $((\gh_{1;\mathbb{O}})) + \R^{7\times 2}$ \\ \hline
14 & $U(1)\times Spin(7)$ & $\C^7 \oplus \R^8$ & $\R$ &
        $((\gh_{7;\C})) + \R^8$ \\ \hline
15 & $U(1)\times Spin(7)$ & $\C^8 \oplus \R^7$ & $\R$ &
        $((\gh_{8;\C})) + \R^7$ \\ \hline
16 & $U(1)\times U(1)\times Spin(8)$ & $\C^8_+\oplus\C^8_-$ & $\R\oplus\R$ &
        $((\gh_{8;\C})) + ((\gh_{8;\C}))$ \\ \hline
17 & $U(1)\times Spin(10)$ & $\C^{16}\oplus\R^{10}$ & $\R$ &
        $((\gh_{16;\C})) + \R^{10}$  \\ \hline
18 & $\{SU(n),U(n),U(1)Sp(\tfrac{n}{2})\}$ & 
        $\C^{n\times 2}\oplus \gs\gu(2)$ & $\R$ &
        $((\gh_{2n;\C})) + \gs\gu(2)$ \\ 
   & \phantom{XXXXX} $\times SU(2)$ & & & \\ \hline
19 & $\{SU(n),U(n),U(1)Sp(\tfrac{n}{2})\}$ & 
        $\C^{n\times 2}\oplus\C^2$ & $\R\oplus\R$ &
        $((\gh_{2n;\C})) + ((\gh_{2;\C}))$ \\ 
   & \phantom{XXXXX} $\times U(2)$ & & & \\ \hline
   & $\{SU(n),U(n),U(1)Sp(\tfrac{n}{2})\}$ & & & \\ 
20 & \phantom{XXXXX}$\times SU(2)\times$ & 
        $\C^{n\times 2}\oplus\C^{2\times m}$ & 
        $\R \oplus \R$ & $((\gh_{2n;\C})) + ((\gh_{2m;\C}))$ \\
   & $\{SU(m),U(m),U(1)Sp(\tfrac{m}{2})\}$ & & & \\ \hline
\end{tabular}
\end{equation}
}
\hfill table continued on next page ....
\newline
\newpage
\addtocounter{equation}{-1}
{\footnotesize
\noindent .... table continued from previous page
\begin{equation} 
\begin{tabular}{|r|l|l|l|l|}\hline
\multicolumn{5}{| c |}{Maximal Indecomposable Principal Saturated Nilpotent 
Gelfand Pairs $(N\rtimes K,K)$,}\\
\multicolumn{5}{| c |}{$N$ Nonabelian, Where the Action of $K$ on 
$\gv \cong \gn/[\gn,\gn]$ is Reducible \quad \cite{Y2005}, \cite{Y2006}}\\
\hline \hline
 & Group $K$ & $K$--module $\gv$ & $K$--module $[\gn,\gn]$ & Algebra $\gn$ 
        \\ \hline
   & $\{SU(n),U(n),U(1)Sp(\tfrac{n}{2})\}$ & & & \\ 
20 & \phantom{XXXXX}$\times SU(2)\times$ & 
        $\C^{n\times 2}\oplus\C^{2\times m}$ & 
        $\R \oplus \R$ & $((\gh_{2n;\C})) + ((\gh_{2m;\C}))$ \\
   & $\{SU(m),U(m),U(1)Sp(\tfrac{m}{2})\}$ & & & \\ \hline
21 & $\{SU(n),U(n),U(1)Sp(\tfrac{n}{2})\}$ & 
        $\C^{n\times 2}\oplus\C^{2\times 4}$ &
        $\R \oplus \R$ & $((\gh_{2n;\C})) + ((\gh_{8;\C})) + \R^6$ \\
   & \phantom{XXXXX} $\times SU(2) \times U(4)$ & \phantom{XXXX}$\oplus \R^6$ & 
        & \\ \hline
22 & $U(4)\times U(2)$ & $\R^6\oplus\C^{4\times 2}$ &
        $\R$ & $\R^6 + ((\gh_{8;\C})) + \gs\gu(2)$ \\ 
   & & \phantom{XXXX}$\oplus\gs\gu(2)$ & & \\ \hline
23 & $U(4)\times U(2)\times U(4)$ & $\R^6\oplus\C^{4\times 2}$ &
        $\R\oplus\R$ & $\R^6+((\gh_{8;\C}))$ \\ 
   & & \hfill $\oplus\C^{2\times 4}\oplus\R^6$ & & 
        \hfill $+((\gh_{8;\C})) +\R^6$ \\ \hline
24 & $U(1)\times U(1)\times SU(4)$ & $\C^4\oplus\C^4\oplus\R^6$ & 
        $\R\oplus\R$ & $((\gh_{4;\C})) + ((\gh_{4;\C})) + \R^6$ \\ \hline
25 & $(U(1)\cdot)SU(4)(\cdot SO(2))$ & $\C^4 \oplus \R^{6\times 2}$ &
        $\R$ & $((\gh_{4;\C})) + \R^{6\times 2}$ \\ \hline
\end{tabular}
\end{equation}
}
\medskip

We go on to describe the representation theory relevant to these tables.

\section{Square Integrable Representations}
\label{sec3}
\setcounter{equation}{0}

A connected simply connected Lie group $N$
with center $Z$ is called {\em square integrable} if it has unitary
representations $\pi$ whose coefficients $f_{u,v}(x) = 
\langle u, \pi(x)v\rangle$ satisfy $|f_{u,v}| \in L^2(N/Z)$.  
C.C. Moore and the author worked out the structure and representation
theory of these groups \cite{MW1973}.  If $N$ has one 
such square integrable representation then there is a certain polynomial
function $\Pf(\lambda)$ on the linear dual space $\gz^*$ of the Lie algebra of
$Z$ that is key to harmonic analysis on $N$.  Here $\Pf(\lambda)$ is the
Pfaffian of the antisymmetric bilinear form on $\gn / \gz$ given by
$b_\lambda(x,y) = \lambda([x,y])$.  The square integrable
representations of $N$ are the $\pi_\lambda$ (corresponding to coadjoint
orbits $\Ad^*(N)\lambda$) where $\lambda \in \gz^*$ with $\Pf(\lambda) \ne 0$,
Plancherel  almost irreducible unitary representations of $N$ are square
integrable, and up to an explicit constant 
$|\Pf(\lambda)|$ is the Plancherel density of the unitary
dual $\widehat{N}$ at $\pi_\lambda$.  Concretely,
\medskip

\begin{theorem}{\rm (\cite{MW1973})}\label{sqint-plancherel}
Let $N$ be a connected simply connected nilpotent Lie group that has
square integrable representations.  Let $Z$ be its center.
If $f$ is a Schwartz class
function $N \to \C$ and $x \in N$ then
\begin{equation}\label{inversion}
f(x) = c\int_{\gz^*} \Theta_{\pi_\lambda}(r_xf) |\Pf(\lambda)|d\lambda
\end{equation}
where $c = d!2^d$ with $2d = \dim \gn/\gz$\,, $r_xf$ is the
right translate $(r_xf)(y) = f(yx)$, and $\Theta$ is the distribution
character
\begin{equation}\label{def-dist-char}
\Theta_{\pi_\lambda}(f) = c^{-1}|\Pf(\lambda)|^{-1}\int_{\cO(\lambda)}
        \widehat{f_1}(\xi)d\nu_\lambda(\xi) \text{ for } f \in \cC(N).
\end{equation}
Here $f_1$ is the lift
$f_1(\xi) = f(\exp(\xi))$ of $f$ from $N$ to $\gn$, 
$\widehat{f_1}$ is its classical Fourier transform,
$\cO(\lambda)$ is the coadjoint orbit $\Ad^*(N)\lambda = \gv^* + \lambda$,
and $d\nu_\lambda$ is the translate of normalized Lebesgue measure from
$\gv^*$ to $\Ad^*(N)\lambda$.  
\end{theorem}

The connection with commutative nilmanifolds is

\begin{theorem}\label{mostly-sqint}{\rm (\cite[Theorem 14.4.3]{W2007})}
All of the nilpotent groups in {\rm Table (\ref{vin-table})} are square
integrable except for those of table entries {\rm (1)} and {\rm (6)} with $n$
odd, and those of table entry {\rm (3)}.  All the groups $N$ of
{\rm Table (\ref{vin-table-ipms})} are square integrable.
\end{theorem}

In \cite{W2012} and \cite{W2013} we extended the theory of square integrable 
nilpotent groups to ``stepwise square integrable'' nilpotent groups.
See \cite[Theorem 14.4.3]{W2007}.  Now we settle the three ``exceptional'' 
cases of Theorem \ref{mostly-sqint} by checking that, in those cases, 
the nilpotent group is stepwise square integrable in a straightforward way.
{\sc Case} n refers to table entry (n) in Table (\ref{vin-table}).
\medskip

\noindent {\sc Case 1:} $\gn = \Lambda^2(\R^n) + \R^n$,\, $n = 2m+1$ odd, 
with composition $[(z,u),(w,v)] = (u\wedge v, 0)$.   Choose a basis
$\{u_1, u_2, \dots , u_{2m-1}, u_{2m}, u_{2m+1}\}$ of $\gv = \R^n$.  Then
$\{u_i\wedge u_j \mid i < j\}$ is a basis of $\gz = \Lambda^2(\R^n)$ and
$\{(u_i \wedge u_j)^* \mid i < j\}$ is the dual basis of $\gz^*$.  Define 
$\gl_1 = \gz + \gv_1$ where $\gv_1 = \Span\{u_1, \dots , u_{2m}\}$ and
$\gl_2 = u_{2m+1}\R$.  Then $\gl_1$ is an ideal in $\gn$\,, $\gn = \gl_1
\supsetplus \gl_2$ semidirect sum, and $L_1 := \exp(\gl_1)$ has square
integrable representations.  For the latter define
$\lambda_a = a_1(u_1\wedge u_2)^* + a_2(u_3\wedge u_4)^* + \dots +
a_m(u_{2m-1}\wedge u_{2m})^*$ where $a \in \R^m$.  Then the bilinear
form $b_{\lambda_a}$ on $\gl_1/\gz$ has matrix, in the basis
$\{u_1, u_2, \dots , u_{2m-1}, u_{2m}\}$ of $\gv_1$\,, given by
$
\diag\left\{ \left ( \begin{smallmatrix} 0 & -a_1 \\ a_1 & 0 
                     \end{smallmatrix} \right ), \dots ,
                   \left ( \begin{smallmatrix} 0 & -a_m \\ a_m & 0 
                     \end{smallmatrix} \right ) \right \}.
$
So in this basis $|\Pf(\lambda_a)| = |a_1a_2\dots a_m|$.  That proves the
square integrability of $L_1$.
\medskip

Now $N = L_1\rtimes L_2$ where $L_2 = \exp(\gl_2) \cong \R$.  This is 
a very simple case of the $2$--step square integrability described in 
\cite{W2012} and \cite{W2013}.  The square integrable representations
$\pi_\lambda$ of $L_1$ extend to representations $\pi'_\lambda$ of
$N$ on the same Hilbert space, and unitary characters $\chi_\xi(exp(tu_n)
= e^{i \xi t}$ of $L_2$ can be viewed as unitary characters on $N$
whose kernel contains $L_1$.  Plancherel measure for $N$ is concentrated
on $\{\pi'_\lambda \boxtimes \chi_\xi \mid \Pf(\lambda) \ne 0 \text{ and }
\xi \in \gl_2^* \}$.  If $x \in N$ denote $x = x_1x_2$ with $x_i \in L_i$\,.
Using Theorem \ref{sqint-plancherel} and the Mackey machine we have

\begin{proposition}\label{plancherel-case1}
If $f$ is a Schwartz class function on $N$ then
\begin{equation}\label{inversion1}
f(x) = \tfrac{1}{\sqrt{2\pi}} m!2^m \int_{\gl_2^*} \left (
	\int_{\gz^*} \Theta_{\pi_\lambda}(r_{x_1}f) |\Pf(\lambda)|d\lambda
	\right ) \chi_\xi(x_2) d\xi
\end{equation}
where $\Theta_{\pi_\lambda}$
is the distribution character of $\pi_\lambda \in \widehat{L_1}$\,.
\end{proposition}
\medskip

\noindent {\sc Case 6:}  $\gn = \Lambda^2(\C^n) + \C^n$,\, $n = 2m+1$ odd,
with composition $[(z,u),(w,v)] = (u\wedge v, 0)$.  Here $\gn$ is the
underlying real Lie algebra of the complexification of the algebra of
Case 1 just above.  So $N = L_1\rtimes L_2$ as before -- complex 
instead of real.  If $a \in \C^m$  and 
$\lambda_a = a_1(u_1\wedge u_2)^* + a_2(u_3\wedge u_4)^* + \dots +
a_m(u_{2m-1}\wedge u_{2m})^*$ then $|\Pf(\lambda_a) = |a_1a_2\dots a_m|^2$, so
$L_1$ is square integrable and $N = L_1\rtimes L_2$ where
$L_2 = \exp(\gl_2) \cong \C \cong \R^2$.  Now as in Case 1,
Plancherel measure for $N$ is concentrated
on $\{\pi'_\lambda \boxtimes \chi_\xi \mid \Pf(\lambda) \ne 0 \text{ and }
\xi \in \R^2\}$.  If $x \in N$ denote $x = x_1x_2$ with 
$x_1 \in L_1$ and $x_2 \in L_2 = \R^2$.  

\begin{proposition}\label{plancherel-case6}
If $f$ is a Schwartz class function on $N$ then
\begin{equation}\label{inversion6}
f(x) = \tfrac{1}{2\pi} (2m)!2^{2m} \int_{\gl_2^*} \left (
        \int_{\gz^*} \Theta_{\pi_\lambda}(r_{x_1}f) |\Pf(\lambda)|d\lambda
        \right ) \chi_\xi(x_2) d\xi
\end{equation}
where $\Theta_{\pi_\lambda}$
is the distribution character of $\pi_\lambda \in \widehat{L_1}$\,.
\end{proposition}
\medskip

\noindent {\sc Case 3:}  $\gn = \Im\mathbb{O} + \Im\mathbb{O}$ with composition
$[(z,u),(w,v)] = (\Im(u\overline{v}),0) = (-\Im(uv),0)$.  Recall the
multiplication table for the octonions:
$\mathbb{O}$ is the algebra over $\R$ with basis $\{e_0 , \dots , e_7 \}$ and 
multiplication defined\footnote{This is one of several standards.  The nice 
thing about standards is that there are so many of them.} by
(a) $e_0e_j = e_j = e_je_0$ for $1 \leqq j \leqq 7$, (b) $e_j^2 = -e_0$
for $1 \leqq j \leqq 7$, (c) $e_je_k + e_ke_j = 0$ for $1 \leqq j,k \leqq 7$
with $j \neq k$, (d) $e_1e_2=e_3$, $e_3e_5=e_6$, $e_6e_7=e_1$, $e_1e_4=e_5$,
$e_3e_4=-e_7$, $e_6e_4=e_2$ and $e_2e_5=e_7$, and (e) each equation in (d)
remains true when the subscripts involved in it are cyclically permuted.
This multiplication table is summarized in the diagram
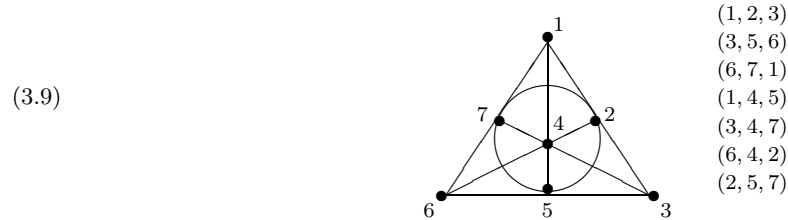
\addtocounter{equation}{1}
$$
\centerline{
\setlength{\unitlength}{.75 mm}
\begin{picture}(70,35)
\put(20,10){\circle{30}} 
\put(20,9){\circle*{2}} 
\put(20,01){\circle*{2}} 
\put(28.5,13){\circle*{2}} 
\put(11.5,13){\circle*{2}} 
\put(2,0){\line(1,0){36}}  
\put(2,0){\line(2,3){18}}  
\put(38,0){\line(-2,3){18}}  
\put(1.2,-0.3){\circle*{2}}   
\put(38.8,-0.3){\circle*{2}}   
\put(20,27.9){\circle*{2}}   
\put(2,0){\line(2,1){26}}  
\put(38,0){\line(-2,1){26}}  
\put(20,00){\line(0,1){28}}  
\put(21,29){\footnotesize{$1$}}
\put(30,13){\footnotesize{$2$}}
\put(40,-4){\footnotesize{$3$}}
\put(21,11.5){\footnotesize{$4$}}
\put(19,-4){\footnotesize{$5$}}
\put(-2,-4){\footnotesize{$6$}}
\put(7.5,13){\footnotesize{$7$}}
\put(50,31){\footnotesize{$(1,2,3)$}}
\put(50,26){\footnotesize{$(3,5,6)$}}
\put(50,21){\footnotesize{$(6,7,1)$}}
\put(50,16){\footnotesize{$(1,4,5)$}}
\put(-75,16){(\theequation)}
\put(50,11){\footnotesize{$(3,4,7)$}}
\put(50,6){\footnotesize{$(6,4,2)$}}
\put(50,1){\footnotesize{$(2,5,7)$}}
\end{picture}
}
$$
\medskip
If $x = x_0e_0 + \dots + x_7e_7 \in \mathbb{O}$ then $\Im x$ is the imaginary 
{\em component} $x_1e_1 + \dots + x_7e_7$\,.
\medskip

Denote $\gl = \gz + \gv$ as before where $\gv$ has basis 
$\{(0,e_1),\dots ,(0,e_7)\}$, the center $\gz$ has basis
$\{(e_1,0), \dots ,(e_7,0)\}$, and $\{(e_1,0)^*, \dots ,(e_7,0)^*\}$ is the 
dual basis of $\gz^*$.
Define $\gv_1 = \Span\{(0,e_1),\dots ,(0,e_6)\}$ and $\gl_1 = \gz + \gv_1$\,.
Note $e_1e_2 = e_3$\,, $e_3e_5 = e_6$\, and $e_6e_4 = e_2$\,.  If $a \in \R^3$
define $\lambda_a = a_1(e_3,0)^* + a_2(e_6,0)^* + a_3(e_2,0)^*$.  In the
ordered basis $\{(0,e_1),(0,e_2),(0,e_3),(0,e_5),(0,e_6),(0,e_4)\}$ of $\gv_1$
the bilinear form $b_{\lambda_a}$ on $\gl_1/\gz$ has matrix
$
\diag\left\{ \left ( \begin{smallmatrix} 0 & -a_1 \\ a_1 & 0 
                     \end{smallmatrix} \right ), 
		\left ( \begin{smallmatrix} 0 & -a_2 \\ a_2 & 0 
                     \end{smallmatrix} \right ),
                   \left ( \begin{smallmatrix} 0 & -a_3 \\ a_3 & 0 
                     \end{smallmatrix} \right ) \right \}.
$
In this basis $|\Pf(\lambda_a)| = |a_1a_2a_3|$.  In particular $L_1
= \exp(\gl_1)$ has square integrable (modulo the center) representations.
Recall the semidirect product structure $N = L_1\rtimes L_2$ where
$L_2 = \exp(e_7\R) \cong \R$.  If $x \in N$ we write $x = x_1x_2$ with 
$x_i \in L_i$\,.  Here 
$c = \tfrac{1}{\sqrt{2\pi}} 3!2^3 = \tfrac{48}{\sqrt{2\pi}}$.  
Now, as for Case 1, 
\begin{proposition}\label{plancherel-case3}
If $f$ is a Schwartz class function on $N$ then
\begin{equation}\label{inversion3}
f(x) = \tfrac{48}{\sqrt{2\pi}} \int_{\gl_2^*} \left (
        \int_{\gz^*} \Theta_{\pi_\lambda}(r_{x_1}f) |\Pf(\lambda)|d\lambda
        \right ) \chi_\xi(x_2) d\xi
\end{equation}
where $\Theta_{\pi_\lambda}$
is the distribution character of $\pi_\lambda \in \widehat{L_1}$\,.
\end{proposition}
\medskip

The structural similarities between Propositions \ref{plancherel-case1},
\ref{plancherel-case6} and \ref{plancherel-case3} are clear.  We summarize them
as

\begin{theorem}\label{plancherel-case163}
Let $X = G/K$ be one of the three ``exceptional'' cases of
{\rm Table \ref{vin-table}}, in other words entry {\rm (1)} or {\rm (6)} with
$n$ odd, or entry {\rm (3)}.  Decompose $N = N_1\rtimes N_2$ as
in {Section \ref{sec3}}.  Let $x\in X$, say $x = g_1g_2K$ where $g_i \in L_i$\,.
If $f$ is a Schwartz class function on $X$ then
\begin{equation}\label{inversion163}
f(x) = \tfrac{2^d d!}{(2\pi)^{\dim \gl_2/2}} 
        \int_{\gl_2^*} \left ( \int_{\gl_2^*}
         \Theta_{\pi_{\Ad^*(k)\lambda}}(r_{g_1}f)
         |\Pf(\lambda)|d\lambda
        \right ) \chi_\xi(g_2) d\xi
\end{equation}
where $d = \tfrac{1}{2}\dim \gl_1 / \gz$ and $\Theta_{\pi_\lambda}$
is the distribution character of $\pi_\lambda \in \widehat{L_1}$\,.
\end{theorem}

We go on to carry these Fourier Inversion formulae over to the nilmanifolds
themselves.

\section{Analysis on the Nilmanifold}
\label{sec4}
\setcounter{equation}{0}
The commutative nilmanifolds $X = (N\rtimes K)/K$ evidently have the 
property that $N$ is simply transitive.  Thus one can interpret the
inversion formulae, from Theorem \ref{mostly-sqint} and Propositions 
\ref{plancherel-case1}, \ref{plancherel-case6} and \ref{plancherel-case3}, 
as Fourier inversion formulae on $X$.
However this ignores the role of the isotropy subgroup $K$, which already
has a big impact in the case of ordinary euclidean space.
\medskip

Choose an $K$--invariant inner product $(\lambda,\nu)$ on $\gz^*$\,. Denote
$\gz^*_t = \{\lambda \in \gz^* \mid (\lambda,\lambda) = t^2\}$, the sphere of
radius $t$.  Consider the action of $K$ on $\gz^*_t$.  Recall that two orbits
$\Ad^*(K)\xi$ and $\Ad^*(M)\nu$ are of the {\sl same orbit type} if the
isotropy subgroups $K_\xi$ and $K_\nu$ are conjugate, and an orbit is
{\em principal} if all nearby orbits are of the same type.  Since $K$ and
$\gz^*_t$ are compact, there are only finitely many orbit types of $K$
on $\gz^*_t$, there is only one principal orbit type, and the union of the
principal orbits forms a dense open subset of $\gz^*_t$ whose complement
has codimension $\geqq 2$.  There are many good expositions of this material,
for example \cite[Chap. 4, \S 3]{B1972} for a complete treatment,
\cite[Part II, Chap. 3, \S 1]{GOV1993} modulo
references to \cite{B1972}, and \cite[Cap. 5]{N2005} for a more basic
treatment but still with some references to \cite{B1972}.
Principal orbit isotropy subgroups of compact connected
linear groups are studied in detail in the work \cite{HH1970}  of
W.-C. and W.-Y. Hsiang, so the possibilities for $K_\lambda$ ($\Ad^*(K)\lambda$
principal) are essentially known.  But it is not to difficult to work
out the principal orbit isotropy subgroups $K_\lambda$ directly in our cases.
\medskip

Define the $\Pf$-nonsingular principal orbit set for $K$ on $\gz^*$ as follows:
\begin{equation}\label{defregset}
\gw^* = \{\lambda \in \gz^* \mid \Pf(\lambda) \ne 0 \text{ and }
	\Ad^*(K)\lambda \text{ is a principal } K\text{-orbit on } \gz^*\}.
\end{equation}
The principal orbit set $\gw^*$ is a dense open set of codimension $\geqq 2$
in $\gz^*$.  If $\lambda \in \gw^*$ and $c \ne 0$ then
$c\lambda \in \gz^*$ with isotropy $K_{c\lambda} = K_\lambda$\,.
There is a Borel section $\sigma$ to
$\gw^* \to \gw^*/\Ad^*(K)$ which picks out an element in each $K$-orbit,
so that $K$ has the same isotropy subgroup at each of those elements.  In
other words in each $K$-orbit on $\gw^*$ we measurably choose an element
$\lambda = \sigma(\Ad^*(K)\lambda)$ such that those isotropy subgroups
$K_\lambda$ are all the same.  Let us denote
\begin{equation}\label{k-diamond}
K_\diamondsuit \text{: isotropy subgroup of } K \text{ at }
\sigma(\Ad^*(K)\lambda) \text{ for every } \lambda \in \gw^*
\end{equation}
Then we can replace $K_\lambda$ by $K_\diamondsuit$, independent of $\lambda
\in \gw^*$, in our considerations.
\medskip

Fix $\lambda = \sigma(\Ad^*(K))\lambda \in \gw^*_t := \gw \cap \gz_t$\,.
Consider the semidirect
product group $N\rtimes K_\diamond$.  We write $\cH_\lambda$ for the
representation space of $\pi_\lambda$\,.
The next step is to extend the representation $\pi_\lambda$ to
a unitary representation $\pi'_\lambda$ of $N\rtimes K_\diamond$
on the same representation space $\cH_\lambda$.  
Following the Fock space argument of \cite[Lemma 3.8]{W2014} we see
that the Mackey obstruction to this extension is trivial.  Thus

\begin{lemma}\label{no-obstruction}
The square integrable representation $\pi_\lambda$ of $N$ extends to an
irreducible unitary representation
$\pi'_\lambda$ of $N\rtimes K_\diamond$ on the
representation space of $\pi_\lambda$\,.
\end{lemma}

We proceed {\em mutatis mutandis} along the lines of \cite[\S 3]{W2014}.
Each $\lambda = \sigma(\Ad^*(K))\lambda \in \gw^*$ defines classes
\begin{equation}\label{nk-lambda-family}
\cE(\lambda) := \left \{\pi_\lambda' \otimes \gamma \mid
\gamma \in \widehat{K_\diamond}\right \}
\text{ and } \cF(\lambda) := \left \{\Ind_{NK_\diamond}^{NK} 
(\pi_\lambda' \otimes \gamma )\mid \pi_\lambda'
\otimes \gamma \in \cE(\lambda)\right \}
\end{equation}
of irreducible unitary representations of $N\rtimes K_\diamond$ and
$N\rtimes K$.  The Mackey little group method, plus the facts that the
Plancherel density on $\widehat{N}$ is polynomial on $\gz^*$
and $\gz^*\setminus\gw^*$ has measure $0$, give us

\begin{lemma}\label{rep-nk}
Plancherel measure for $N\rtimes K$ is concentrated on the set
$\bigcup_{\lambda \in \gw^*}\cF(\lambda)$ of $($equivalence classes
of\,$)$ irreducible representations given by
$\eta_{\lambda,\gamma} := \Ind_{NK_\diamond}^{NK} 
(\pi_\lambda' \otimes \gamma)$ with
$\pi_\lambda' \otimes \gamma \in \cE(\lambda)$ and
$\lambda = \sigma(\Ad^*(K)\lambda \in \gw^*$.  Further
\begin{equation}\label{orb-rep}
\eta_{\lambda,\gamma}|_N = 
\left . \left ( \Ind_{NK_\diamond}^{NK}(\pi_\lambda' 
        \otimes \gamma)\right ) \right |_N = \int_{K/K_\diamond} 
        (\dim \gamma)\, \pi_{\Ad^*(k)\lambda}\, d(kK_\diamond).
\end{equation}
\end{lemma}

The open subset $\gw^*$ of $\gz^*$ fibers over $\gw^*/\Ad^*(K)$ with
compact fiber $K/K_\diamond$\,.  So the euclidean measure on $\gz*$
pushes down to a measure $\overline{\mu}$ on $\gw^*/\Ad^*(K)$.
Taking into account Lemma \ref{rep-nk} and the Peter--Weyl Theorem for 
$K_\diamond$ this gives us
\begin{equation}\label{planchNK}
L^2(N\rtimes K) = \sum_{\gamma \in \widehat{K_\diamond}}\deg(\gamma)
\int_{\gw^*/\Ad^*(K)} (\cH_{\eta_{\lambda,\gamma}}\widehat{\otimes}
  \cH_{\eta_{\lambda,\gamma}}^*)|\Pf(\lambda)|d\overline{\mu}(\Ad^*(K)\lambda).
\end{equation}

In order to push (\ref{planchNK}) down to the commutative space $(N\rtimes K)/K$
we need

\begin{lemma}\label{sph-reps}
The representation $\eta_{\lambda,\gamma}$ of {\rm (\ref{orb-rep})} has a
$K$--fixed unit vector if and only if $\gamma$ is the trivial representation
of $K_\diamond$\,.
\end{lemma}

\begin{proof} We take the representation space of $\pi_\lambda$ to be
the space of Hermite polynomials $p(z)$ on $\gv = \C^m$.  There $K_\lambda$
acts as a subgroup of $U(m)$.  So $p(z) = 1$ is a fixed unit vector
for $\pi'_\lambda(K_\diamond)$.  Let $1_\diamond$ denote the trivial
representation of $K_\diamond$\,.    Now $\eta_{\lambda,1_\diamond}$ has a
$K$--fixed unit vector.  Conversely if $\eta_{\lambda,\gamma}$ has a
$K$--fixed unit vector, then $\Ind_{K_\diamond}^K(\gamma)$ has a
$K$--fixed unit vector.  Then, by Frobenius Reciprocity for $K$,  
$\gamma$ is a subrepresentation of the $K_\diamond$--restriction of the
trivial representation of $K$.  Thus $\gamma = 1_\diamond$\,.
\end{proof}

Since $K$ is compact we can understand the Schwartz space of 
$G = N\rtimes K$ as
$$
\cC(G) = \{f \in C^\infty(G) \mid f(\cdot\, k) \in \cC(N) \text{ for each }
	k \in K\}.
$$
Combine Theorems \ref{sqint-plancherel} and \ref{mostly-sqint} with 
(\ref{planchNK}) and Lemma \ref{sph-reps}, and average over $K$, to see

\begin{theorem}\label{sqint-plancherel-ws} 
Let $(X = N\rtimes K)/K$ be a commutative
nilmanifold from {\rm Table \ref{vin-table}} --- except 
{\rm (1)} or {\rm (6)} with $n$ odd or {\rm (3)} --- or
from {\rm Table \ref{vin-table-ipms}}.  Then
$
L^2(X) = \int_K \int_{\gw^*} \cH_{\eta_{\Ad^*(k)\lambda,1_\diamond}} 
	d\lambda \, dk.
$
If $f \in \cC(G)$ and $x \in X$, say $x = gK$ where $g \in N$, then
$$
f(x) = d!\, 2^d \int_{K/K_\diamond} \left ( \int_{\gw^*/\Ad^*(K)}
	\Theta_{\pi_{\Ad^*(k)\lambda}}(r_gf)\, |\Pf(\lambda)|
	d\overline{\mu}(\Ad^*(K)\lambda) \right )  d(kK_\diamond)
$$
where $d = \tfrac{1}{2}\dim \gn/\gz$\,, $r_gf$ is the
right translate $(r_gf)(hK) = f(ghK)$ for $h \in N$, and $\Theta$ 
denotes distribution character as in {\rm Theorem \ref{sqint-plancherel}}.
\end{theorem}
The argument of Lemma \ref{sph-reps} exhibits the spherical function in
each of the $\cH_{\eta_{\lambda,1_\diamond}}$ of Theorem \ref{mostly-sqint}.
\medskip

Now we examine the three exceptional cases of Section \ref{sec3}.  
Retain the notation of Section \ref{sec4}.   These cases are a bit more 
delicate because $\Ad^*(K)$ moves the factors in the semidirect product
decomposition $N = L_1\rtimes L_2$\,.
\medskip

{\sc Case (1).}
$K = SO(2m+1)$ so $\gw^*$ consists of the $SO(2m+1)$--images of the
$\lambda_a$, $a \in \R^m$ with $a_1a_2\dots a_m \ne 0$, given by 
$a_1(u_1 \wedge u_2)^* + a_2(u_3 \wedge u_4)^* + \dots +
a_m(u_{2m-1} \wedge u_{2,})^*$.  
This follows from the Darboux normal form of antisymmetric bilinear forms.
So $\gm^*/\Ad^*(K) \subset \Lambda^2(\R^{2m+1})/SO(2m+1)$ has
representatives $\lambda_a$ as above with 
$a_1 \leqq a_2 \leqq \dots \leqq a_m$\,, or, up to measure $0$,
$a_1 < a_2 < \dots < a_m$\,.  
\medskip

{\sc Case (6).}
$K = SU(2m+1)$ so $\gw^*$ consists of the $SU(2m+1)$--images of the
$\lambda_a$, $a \in \C^m$ with $a_1a_2\dots a_m \ne 0$, given by
$a_1(u_1 \wedge u_2)^* + a_2(u_3 \wedge u_4)^* + \dots +
a_m(u_{2m-1} \wedge u_{2,})^*$,
and one can choose representatives with 
$a_1 \leqq a_2 \leqq \dots \leqq a_{m-1}$ real and $a_{m-1} \leqq |a_m|$.
\medskip

{\sc Case (3).} 
$K$ is the exceptional simple group $G_2$ so 
$\gw^*$ consists of all $G_2$--images of the
$\lambda_a = a_1(e_3,0)^* + a_2(e_6,0)^* + a_3(e_2,0)^*$, 
$a \in \R^3$ and $a_1a_2a_3 \ne 0$.  The point is that, inside $\Im \mathbb{O}$,\, 
$e_2$ and $e_3$ are orthogonal unit vectors that generate a subalgebra 
$\H \subset \mathbb{O}$, and $e_6$ is a unit vector orthogonal to that subalgebra.
\medskip

We now proceed as in the transition from Theorem \ref{mostly-sqint} to
Theorem \ref{sqint-plancherel-ws}, by averaging over $K$.  
We write $\cO_{K,\lambda}$ for the orbit $\Ad^*(K)\lambda$ in $\gw^*$.
In view of Propositions \ref{plancherel-case1}, 
\ref{plancherel-case6} and \ref{plancherel-case3} we arrive at
\medskip

\begin{theorem} \label{sqint-plancherel-ws-163}
Let $X = G/K$ be one of the three ``exceptional'' cases of 
{\rm Table \ref{vin-table}}, in other words entry {\rm (1)} or {\rm (6)} with
$n$ odd, or entry {\rm (3)}.  Decompose $N = N_1\rtimes N_2$ as
in {\rm Section \ref{sec3}}.  Let $x \in X$, say $x = gK$ where $g \in N$.
If $f$ is a Schwartz class function on $X$ then
\begin{footnotesize}
\begin{equation}\label{inversion163-wc}
\begin{aligned}
&f(x) = \\
&\tfrac{2^d d!}{(\sqrt{2\pi})^{d_2}}
	\int_{K/K_\diamond} \left ( \int_{\gl_2^*} \left (
	\int_{\gw^*/\Ad^*(K)} \Theta_{\pi_{\Ad^*(k)\lambda}}(r_{g_{1,k}}f)
	|\Pf(\lambda)|d\overline{\mu}(\cO_{K,\lambda}) \right )
	 \chi_{\Ad^*(k)\xi}(g_{2,k}) d\xi \right ) d(kK_\diamond)
\end{aligned}
\end{equation}
\end{footnotesize}
\hskip -0.15cm 
where $x = g_{1,k}g_{2,k}K$ with $g_{i,k} \in \Ad(k^{-1})L_i$, 
$d = \tfrac{1}{2}\dim \gl_1 / \gz$\,, $d_2 = \dim\gl_2$\,, 
and $\Theta_{\pi_{\Ad^*(k)\lambda}}$
is the distribution character of the representation $\pi_{\Ad^*(k)\lambda}$
of $\Ad(k^{-1})L_1$\,.
\end{theorem}

\section{Weakly Symmetric Spaces}
\label{sec5}
\setcounter{equation}{0}

Let $(X,ds^2)$ be a connected Riemannian manifold and $I(X,ds^2)$ 
its isometry group.  Recall that $(X,ds^2)$ is {\em symmetric} if,
given $x \in X$, there is an isometry $s_x \in I(X,ds^2)$ such
that $s_x(x) = x$ and $ds_x(\xi) = -\xi$ for every tangent vector
$\xi \in T_x(X)$ at $x$.  In other words, $(X,ds^2)$ is symmetric
just when, for every $x \in X$, there is an isometry that fixes
$x$ and reverses
every geodesic through $x$.  Then $s_x$ is the {\em symmetry} at $x$.
We are going to discuss an extension of this notion, called ``weak 
symmetry'', and then make some comments about weakly symmetric 
Finsler manifolds.
\medskip

Again let $(X,ds^2)$ be a connected Riemannian manifold and $I(X,ds^2)$ 
its isometry group.  Suppose that, given $x \in X$ and $\xi \in T_x(X)$,
that there is an isometry $s_{x,\xi} \in I(X,ds^2)$ such that 
$s_{x,\xi}(x) = x$ and $ds_{x,\xi}(\xi) = -\xi$.  Then $(X,ds^2)$ is
{\em weakly symmetric}.  Here note that $s_{x,\xi}$ depends on $\xi$
as well as $x$.  Obviously, symmetric Riemannian manifolds are weakly
symmetric.  As in the symmetric case, one composes symmetries along
a geodesic to see that  weakly symmetric Riemannian manifolds are
complete, in fact homogeneous.  In terms of the Lie group structure
$X = G/K$ where $G = I(X,ds^2)$ (or the identity 
component $I(X,ds^2)^0$), this is equivalent to the existence of an
automorphism $\sigma : G \to G$ such that $\sigma(g) \in Kg^{-1}K$ for
every $g \in G$.  Then $(G,K)$ is a {\em weakly symmetric pair} with
{\em weak symmetry} $\sigma$.  More generally one can make these
definitions for any pair $(G,K)$ where $G$ is a Lie group, $K$ is a 
compact subgroup and $G/K$ is connected.  The connection with this 
note is
\medskip

\begin{theorem}\label{ws-riem}{\rm \cite{S1956}}.  Every weakly symmetric 
pair is a Gelfand pair.  In other words every weakly symmetric Riemannian
manifold is a commutative space.
\end{theorem}

The converse doesn't quite hold.  If $(G,K)$ is a Gelfand pair with
$G$ reductive then $(G,K)$ has a weak symmetry.  On the other hand
there are commutative nilmanifolds that are not weakly symmetric, for
example \cite{L1999} the pair $(\{ (\Re \H^{n\times n}_0 \oplus \Im \H) + \H^n
\} \rtimes Sp(n),Sp(n))$ in Table \ref{vin-table}.  Nevertheless,

\begin{proposition}\label{ws-table} All the commutative pairs of 
{\rm Table \ref{vin-table}} are weakly symmetric except for item {\rm (9)}.
\end{proposition}
Thus Theorems \ref{sqint-plancherel-ws} and \ref{sqint-plancherel-ws-163}
apply to weakly symmetric nilmanifolds $G/K = (N\rtimes K)/K$ where $K$
acts irreducibly on the tangent space.  See \cite{Y2005} and \cite{Y2006},
or the exposition in \cite[Chapter 15]{W2007}, for a complete (modulo
some combinatorics) classification of the weakly symmetric nilmanifolds.
\medskip

All this holds for Finsler manifolds.  If $(X,F)$ is a weakly symmetric 
(using the same definition as in the Riemannian case) Finsler manifold, 
then it is geodesically complete and
homogeneous, say $X = G/K$ where $G = I(X,ds^2)$ (or the identity
component $I(X,ds^2)^0$).  Also, $X$ has a $G$--invariant weakly
symmetric Riemannian metric \cite[Theorem 2.1]{D2011}.  Thus the classification
of weakly symmetric pairs $(G,K)$ is the same for Finsler manifolds as for
Riemannian manifolds, and Theorem \ref{ws-riem}, Proposition \ref{ws-table} and
the remarks after Proposition \ref{ws-table} hold for weakly
symmetric Finsler nilmanifolds.

\medskip
\noindent Department of Mathematics, University of California,\hfill\newline
\noindent Berkeley, California 94720--3840, USA\hfill\newline
\smallskip
\noindent {\tt jawolf@math.berkeley.edu}


\begin{thebibliography}{XX}

\bibitem{B1972} 
G. Bredon,
Introduction to Compact Transformation Groups, Academic Press, 1972.

\bibitem{D2011}
S. Deng,
On the classification of weakly symmetric Finsler spaces,
Israel J. Math. {\bf 181} (2011), 29--52.

\bibitem{GOV1993}
V. V. Gorbatsevich, A. L. Onishchik \& E. B. Vinberg,
Foundations of Lie Theory and Lie Transformation Groups,
Springer, 1997.

\bibitem{HH1970}
W.-C. Hsiang \& W.-Y. Hsiang,
Differentiable actions of compact connected classical groups II,
Annals of Math. {\bf 92} (1970), 189--223.

\bibitem{L1999}
J. Lauret,
Modified $H$--type groups and symmetric--like riemannian spaces,
Diff. Geom. Appl. {\bf 10} (1999), 121--143.


\bibitem{MW1973} C. C. Moore \& J. A. Wolf,
Square integrable representations of nilpotent groups.
Transactions of the American Mathematical Society,
{\bf 185} (1973), 445--462.

\bibitem{N2005}
S. de Neymet Urbina (con la colaboraci\' on de Rolando Jim\' enez B.),
Introducci\' on a los Grupos Topol\' ogicos de Transformaciones,
Sociedad Math\' ematica Mexicana, 2005.

\bibitem{S1956}
A. Selberg,
Harmonic analysis and discontinuous groups in weakly symmetric
riemannian spaces, with applications to Dirichlet series,
J. Indian Math. Soc. {\bf 20} (1956), 47--87.

\bibitem{T1984}
E. F. G. Thomas,
An infinitesimal characterization of Gel'fand pairs,
Contemp. Math. {\bf 26} (1984), 379--385.

\bibitem{V2001}
E. B. Vinberg,
Commutative homogeneous spaces and co--isotropic symplectic actions,
Russian Math. Surveys {\bf 56} (2001), 1--60.

\bibitem{V2003}
E. B. Vinberg,
Commutative homogeneous spaces of Heisenberg type, Trans Moscow Math. Soc.
{\bf 64} (2003), 45--78.


\bibitem{W1963}
J. A. Wolf,
On locally symmetric spaces of non--negative curvature and certain
other locally homogeneous spaces,
Comm. Math. Helv. {\bf 37} (1963), 266--295.

\bibitem{W2007} J. A. Wolf,
Harmonic Analysis on Commutative Spaces.  Math. Surveys \&
Monographs, vol. 142, American Mathematical Society, 2007.

\bibitem{W2012} J. A. Wolf, 
Plancherel Formulae associated to Filtrations of Nilpotent Lie Groups,
\{arXiv: 1212.1908 (math.RT; math.DG, math.FA)\}

\bibitem{W2013} J. A. Wolf,
Stepwise Square Integrable Representations of Nilpotent Lie Groups,
Mathematische Annalen {\bf 357} (2013), 895--914. 

\bibitem{W2014} J. A. Wolf,
The Plancherel Formula for minimal parabolic subgroups, J. Lie Theory
{\bf 24} (2014), 791-808. \{arXiv: 1306.6392 (math RT)\}

\bibitem{Y2005}
O. S. Yakimova,
``Gelfand Pairs,'' Bonner Math. Schriften (Universit\" at Bonn)
{\bf 374}, 2005.

\bibitem{Y2006}
O. S. Yakimova,
Principal Gelfand pairs,
Transformation Groups {\bf 11} (2006), 305--335.


\end{thebibliography}
\enddocument
\end